\documentclass[letterpaper, 10 pt, conference]{ieeeconf}  

\usepackage{algorithm} 
\usepackage{algorithmic} 
\usepackage{amsmath}
\usepackage{amssymb}
\usepackage{mathtools}
\usepackage{amsthm}
\usepackage{bm}
\usepackage{makecell}
\usepackage{xcolor}
\renewcommand\fbox{\fcolorbox{yellow}{yellow}}
\usepackage[hidelinks]{hyperref}

\theoremstyle{plain}
\newtheorem{theorem}{Theorem}
\newtheorem{proposition}[theorem]{Proposition}
\newtheorem{lemma}[theorem]{Lemma}

\theoremstyle{definition}
\newtheorem{definition}[theorem]{Definition}
\newtheorem{assumption}[theorem]{Assumption}
\theoremstyle{remark}

\newcommand{\blue}[1]{{\color{black}{#1}}}

\title{\LARGE \bf
Iterative implicit gradients for nonconvex optimization with variational inequality constraints}

\author{Harshal D. Kaushik and Ming Jin 
\thanks{Bradley Department of Electrical and Computer Engineering, Virginia Tech. 
          { \tt\small harshal.kaushik@utdallas.edu, }%
         {\tt\small jinming@vt.edu}}
}

\begin{document}

\maketitle
\thispagestyle{empty}
\pagestyle{empty}

\begin{abstract}
We propose an optimization proxy in terms of iterative implicit gradient methods for solving constrained optimization problems with nonconvex loss functions. This framework can be applied to a broad range of machine learning settings, including meta-learning, hyperparameter optimization, large-scale complicated constrained optimization, and reinforcement learning. The proposed algorithm builds upon the iterative differentiation (ITD) approach. We extend existing convergence and rate analyses from the bilevel optimization literature to a constrained bilevel setting, motivated by learning under explicit constraints. Since solving bilevel problems using first-order methods requires evaluating the gradient of the inner-level optimal solution with respect to the outer variable (the implicit gradient), we develop an efficient computation strategy suitable for large-scale structures. Furthermore, we establish error bounds relative to the true gradients and provide non-asymptotic convergence rate guarantees.
\end{abstract}

\section{Introduction}
Motivated by applications in meta-learning, hyperparameter optimization, large-scale optimization, and reinforcement learning, this work focuses on a constrained variant of the bilevel optimization problem—a class of problems characterized by equilibrium constraints. Classical formulations of meta-learning and hyperparameter optimization in the existing literature typically overlook such constraints. However, incorporating constraints into learning frameworks is crucial for ensuring safety, fairness, and adherence to high-level specifications \cite{abdeen2022learning, ChamonPaternain2022}.

The key motivation stems from the observation that meta-learning can naturally be expressed as a bilevel optimization problem: the inner-level optimization captures task-specific adaptation, while the outer-level optimization can be augmented with safety constraints to prevent biased or risky decision-making during meta-training. Although bilevel optimization methods have achieved consistent convergence in unconstrained settings \cite{ShedivatBansalBurda2017, HuismanRijnPlaat2021, FinnRajeswaranKakadeLevine2017, raghu2019rapid, FranceschiFrasconiSalzoPonti2018}, optimization-based meta-learning remains computationally demanding. In particular, backpropagation through the solution of the inner-level problem entails high memory usage and complex derivative computations, making it challenging to scale to medium- or large-scale datasets.



A common approach to addressing bilevel optimization problems is to reformulate them as constrained optimization problems \cite{HansenJaumardSavard1992, Moore2010, ShiLuZhang2005, HDK_2016, HDK_TAC, HDK_thesis, HDK2019, HDK2020, HDK2021, HDK2021_Jayesh}, and subsequently solve them by differentiating through the Karush–Kuhn–Tucker (KKT) conditions. Consider, for instance, a linearly constrained optimization problem defined as follows:
    \begin{align}\label{prob:linear_constrianed_opt}
        \underset{x}{\text{minimize} }\left\{\left<\hat c,x\right> \ |\ Ax \leq b, S x \leq t\right\}
    \end{align}
    where $x\in \mathbb{R}^n,$ and $A, S \in \mathbb{R}^{m\times n}$ are the known constraint matrices. For this single leveled, linear optimization problem, the calculation of implicit gradient is  not straightforward. 
In the evaluation of the implicit gradient above, there are certain challenges in the implementation, such as the inversion matrix $H$ becomes increasingly difficult as the number of constraints grows. Currently, there are no reliable approximation techniques available to simplify this inversion step. Moreover, at each iteration, certain constraint qualification conditions must be satisfied to ensure that the matrix $H$ to be invertible. 

\section{Problem  Formulation}

In this work, we consider a constrained bilevel optimization problem with a nonconvex loss function. Specifically, our focus lies on optimization problems governed by variational inequality (VI) constraints. The outer level corresponds to a nonconvex optimization problem, while the inner level is formulated as a variational inequality problem. The VI framework serves as a powerful modeling tool capable of representing a wide range of complex systems, including: i) complementarity problems, ii) noncooperative games, and iii) large-scale, intricately constrained optimization problems \cite{HDK2023VT, HDK_AAAI, HDK2022, HDK2019, HDK2020, HDK2021, HDK_SIAM, HDK_TAC, HDK2025_iScience, HDK2025_Nate, HDK2025PESGM, HDK2024, HDK2025_Cessna, HDK2025Maritime}.
Consider a set valued map $Y(x)\subseteq\mathbb{R}^n$ and mapping $F(\cdot,x):\mathbb{R}^n\to \mathbb{R}^n$ then SOL$(Y(x),F(\cdot,x))$ is the solution of variational inequality VI$(Y(x),F(\cdot, x))$. Let us first define  SOL$(Y(x),F(\cdot,x))$ in the following 
\begin{align*}
    \text{SOL}&(Y(x),F(\cdot,x))  \\= &\left\{ y\in Y(x) : \left<F(y,x),z-y\right> \geq 0\  \text{ for all }  z \in Y \right\}. 
\end{align*}
We consider the following problem formulation
\begin{align}\label{prob:main_prob}
    &\underset{x\in X}{\text{minimize }}  f(y^*(x), x) \tag{P} \\
    &\text{subject to \ }{y^*(x)\in\text{SOL}\left(Y(x),F(\cdot,x)\right)}. \nonumber
\end{align}

For a case when the inner-level of problem \eqref{prob:main_prob} is an unbounded optimization problem, 
there exists a broad collection of approaches, broadly summarized into two categories: \emph{(1)} Iterative implicit differentiation  and \emph{(2)} Approximate implicit differentiation \cite{GhadimiWang2018, KaiyiJiYangLiang2021, GrazziFranceschiPontilSalzo2020, HDK2023VT}.
   
Different from  the existing literature for the  analysis of bilevel optimization problems, in this work, we utilize  a merit function that characterizes the solution of the inner-level VI. In Section \ref{sec:preliminaries}, we  introduce the concept of a D-gap function. This  is useful in characterizing the solution of  the inner-level VI$(Y, F)$ in problem \eqref{prob:main_prob}. Later in Section \ref{sec:converg_bound_and_rates}, we extend the analysis provided in \cite{GrazziFranceschiPontilSalzo2020, KaiyiJiYangLiang2021}   to  optimization problems with variational inequality constraints.

  {\bf Contribution}: Our contribution is summarized as follows:
   
   \noindent \noindent\emph{(1)} In this work, we circumvent the need for backpropagation through the inner-level problem when using the implicit gradient. We propose an optimization proxy where we compute the implicit gradient using concepts from merit functions (specifically, the D-gap function) and fixed-point formulations associated with the natural map of the variational inequality, as detailed in Section~\ref{sec:preliminaries}. The proposed approach is more general and computationally efficient compared to existing methods \cite{FeberWilderDilkinaTambe2020, DontiAmosKolter2017, AgrawalAmosBarrattBoydDiamondKolter2019, barratt2018differentiability}.
   
   \noindent\emph{(2)} In this work, we address a constrained optimization problem \eqref{prob:main_prob}, in contrast to the unconstrained bilevel formulations commonly studied in the literature \cite{GrazziFranceschiPontilSalzo2020, KaiyiJiYangLiang2021}. Specifically, we focus on a class of nonsmooth optimization problems subject to variational inequality (VI) constraints, as described in \eqref{prob:main_prob}. Notably, bilevel optimization emerges as a special case of this broader formulation \cite{HDK_SIAM}.
   
   \noindent\emph{(3)} We extend the analytical framework developed in \cite{GrazziFranceschiPontilSalzo2020} and \cite{KaiyiJiYangLiang2021} to a broader class of optimization problems involving variational inequality (VI) constraints. Specifically, we derive error bounds for the implicit gradients and for the gradients of the objective function with respect to the true gradients. Furthermore, we establish non-asymptotic convergence rate results for the proposed algorithmic scheme.

{\bf Notation.} For the sake of brevity, some places we  write vector $y(x)$ as $y$, set valued map $Y(x)$ as $Y$, and mapping $F(\cdot, x)$ as $F(\cdot)$ or simply $F$. For convenience, \blue{Jacobian of mapping $F:\mathbb{R}^n\to \mathbb{R}^n$ with respect to $x \in \mathbb{R}^m $ at any $y\in \mathbb{R}^n$ is denoted with a bold ${\hm{\nabla}}_x F(y) \in \mathbb{R}^{n \times m}.$  For any $f:\mathbb{R}^n\to\mathbb{R}$, we use $\nabla_x f \in \mathbb{R}^{n}$ to denote a partial derivative and $\nabla_x^2 f$ denotes the Hessian matrix of $f$.} For convenience, instead of SOL$(Y(x),F(\cdot, x))$, some places we alternatively refer the inner-level solution set  by  $S(x)$. For denoting the projection of $x$ onto set $X$, we use $\mathcal{P}_X\{ x\}$. All the norms are Euclidean for vectors and spectral norm for matrices, unless otherwise specified. We represent the inner product between two vectors by $\left<\cdot, \cdot\right>$, whereas for matrices it correspondingly becomes a Frobenius inner product.

Next we provide  necessary assumptions on the problem structure.
\begin{assumption}\label{assum:problem_struct}
    Consider problem \eqref{prob:main_prob}. We have the following hold on the  problem structure: \\
    (a) For  $x\in X\subseteq\mathbb{R}^m$ and $y\in Y(x) \subseteq\mathbb{R}^n$,  the outer objective function $f(x, y)$ is  continuously differentiable with respect to  $x$ and $y$.\\ 
    (b) For any $x\in X$ and $y\in Y(x)$, the inner-level map $F(\cdot, x):\mathbb{R}^n\to \mathbb{R}^n$ is  continuously differentiable and $\mu$-strongly monotone  with respect to   $y\in Y$. \\
    (c)  Set $X$ and for any $x\in X$, the set $Y(x)  $ are closed, convex, and bounded.\\ 
 \end{assumption}
 Next, we provide assumption on the set $Y$ such that necessary constraint qualification conditions hold. The following assumption comes handy in establishing the continuity of the solution map ($S(x)$) of the inner-level VI in Lemma \ref{lem:piecewise_smooth_S(x)}.
\begin{assumption}\label{assum:implicit_grad_exactness}
Consider problem \eqref{prob:main_prob}. For the inner-level VI$(Y,F)$, with functional map $Y(x) \equiv \{y\in \mathbb{R}^m : g_i(x,y) \leq 0 \}$ such that for a feasible point $(\bar x,\bar y)$, we have:
\begin{itemize}
    \item[(a)] There exists vector $v\in \mathbb{R}^m$ such that $\blue{\left<v,\nabla_y g_i(\bar x, \bar y)\right>} <0, \text{ for all } i\in \mathcal{I}(\bar x)$ where $\mathcal{I}(\bar x) \triangleq \{i: g_i(\bar x, \bar y) = 0\}.$
    \item[(b)] Consider a neighborhood $ W$ of $(\bar x, \bar y)$. The rank of gradient vectors $\{ \nabla_yg_i(x,y) : i \in \mathcal{I}(\bar x) \}$ is constant for any $(x,y)$ in W.
    \item[(c)] The gradient matrix  $\{ \nabla_yg_i(x,y) : i \in \mathcal{I}(\bar x) \}$ has a full-row rank.
    \item[(d)] The matrix formed using Lagrangian $L(\bar x, \bar y, \bar \lambda)$, $\left<U,L(\bar x, \bar y, \bar \lambda),U\right>$ is nonsingular where $U$ is the orthogonal basis of the null space of $\nabla_yg_{\mathcal{I}(\bar x)}(x,y)$. 
\end{itemize}
\end{assumption}


Problem \eqref{prob:main_prob} can be addressed by the first-order  method. A general outline for iteration $k$ is 
\begin{align}\label{alg:generic_first_order_scheme}
    x_{k+1} := x_k - \gamma \blue{\nabla_x} f(y^*(x_k), x_k),
\end{align}
where $\nabla_x f(x)$ is $L$-continuous and $\gamma<1/L$. In the above, calculation of the gradient  $\nabla_x f(y^*(x_k), x_k)$ involves the implicit gradient $\blue{\hm{\nabla}_x y^*(x_k)}$, which can be challenging to estimate. Instead, an approximate version can be obtained from \emph{(1)} ITD, and \emph{(2)} Approximate iterative differentiation (AID). In this paper, we extend the  analysis of  ITD \cite{GrazziFranceschiPontilSalzo2020, KaiyiJiYangLiang2021}  for a class of optimization problems \eqref{prob:main_prob} with VI constraints. 

    \section{Estimation of Implicit Gradient}\label{sec:preliminaries}
    In this section,  we will provide  necessary preliminary results and  obtain the implicit gradient in Theorem \ref{thm:differentiation_D-gap}.  We begin this section with a metric to characterize the optimality of a solution to the inner-level VI problem in problem \eqref{prob:main_prob}.

\begin{definition}\label{def:D_gap_theta}
Consider problem \eqref{prob:main_prob} and let Assumption \ref{assum:problem_struct}(b, c) hold on mapping $F(\cdot, x)$ and set $Y(x)$, respectively. 
For scalars $b>a>0,$  $y\in \mathbb{R}^n$, the merit function of $\phi_{ab}(y,x)$ is defined as 
$$\phi_{ab}(y,x)\triangleq \phi_a(y,x) - \phi_b(y,x), $$
 where for  any  $c>0$ and a positive definite matrix $G$,  $\phi_c(y,x)$ is as follows
    \begin{align}\label{prob:project_y_(c,theta)}
        \phi_c(y,x) &\triangleq \underset{z\in Y}{\text{sup}} \left\{\left<F(y,x),y-z\right> \right.\nonumber\\&\qquad\left.- \frac{c}{2}\left<y - z,G,y-z\right> \right\}.
    \end{align}    
 \end{definition}

 In the next result, we list an important property of  $\phi_{ab}$, that will be used to characterize the root point. 
 \begin{lemma}[\cite{facchinei2003finite}]\label{lem:prop_paramtric_D_gap}
     Consider problem \eqref{prob:main_prob} and let the merit function  $\phi_{ab}(y, x)$ be  given by Definition \ref{def:D_gap_theta} for any $y\in Y$ and $x\in X$. Then  the root  point $y_s\in Y$ of  $\phi_{ab}(y,x)$ (i.e. solution to  $\phi_{ab}(y_s, x) = 0$)  also solves $\mathrm{VI}(Y(x),F(\cdot, x))$ and $y_s\in \mathrm{SOL}(Y(x),F(\cdot, x))$. 
 \end{lemma}
 
 In the next result, we will show that the inner-level solution of VI can be neatly obtained by solving a fixed-point equation. Further, we will  obtain the implicit gradient.

 \begin{theorem} \label{thm:differentiation_D-gap}
    Consider problem \eqref{prob:main_prob}. Let Assumption \ref{assum:problem_struct} (b, c) hold on map ${F(\cdot, x)}$ and set {$Y(x)$}, respectively. Let $y_s{\in Y} $ be a solution of the inner-level variational inequality problem,  i.e. $ y_s\in \mathrm{SOL}(Y(x), F(\cdot, x))$. Then 
    \begin{itemize}
        \item[(a)] For a scalar $b>0$, we have $y_s = z^*_b(y_s,x).$
        \item[(b)] for scalars $b>a>0$, we obtain the implicit  gradient $\blue{\hm{\nabla}_x y}$ at $y_s$  as follows
         \begin{align}\label{eqn:grad_theta_1}
      \hspace{-0.25cm} \blue{\hm{\nabla}_x y} = \left<\underbrace{\blue{\hm{\nabla}_y z_b^*(y, x)}}_{\text{term 1}}, \blue{\hm{\nabla}_x y}\right> +  \underbrace{ \blue{\hm{\nabla}_x z_b^*(y, x)}}_{\text{term 2}},
    \end{align}
     where $z_c^*(y, x)$ is the optimal solution of $\underset{z\in Y}{\text{sup}}\left\{F(y, x)^T(y-z) - \frac{c}{2}\|y-z\|^2\right\}$ and terms 1,  2 can be obtained from differentiating through optimization problem \eqref{prob:project_y_(c,theta)} \cite{AgrawalAmosBarrattBoydDiamondKolter2019}
    \end{itemize}
\end{theorem}
 \begin{proof}
       For any point $y\in Y$, from Definition \ref{def:D_gap_theta} and taking $G$ as an identity matrix, we have the following
            \begin{align}\label{eqn:thm1_bound}
                \phi_{ab}(y,x) = & \phi_a(y,x) - \phi_b(y,x) \nonumber\\
                =&\underset{z\in Y}{\text{sup}} \left\{\left<F(y, x),y-z\right> - \frac{a}{2}\|x - z\|^2 \right\}\nonumber\\&- \underset{z\in Y}{\text{sup}}  \left\{\left<F(y, x),y-z\right> - \frac{b}{2}\|y - z\|^2 \right\}.
            \end{align}
            Let us now consider $z^*_c(y,x)$ as the unique optimal solution of $\underset{z\in Y}{\text{sup}} \left\{F(y, x)^T(y-z) - \frac{c}{2}\|y - z\|^2 \right\}$ for  $c>0$.  Therefore, we can now bound equation \eqref{eqn:thm1_bound} as the following
            \begin{align}\label{eqn:thm1_bound}
                \phi_{ab}(y,x) =&  \left<F(y,x),y-z^*_a(y,x)\right> - \frac{a}{2}\|y - z^*_a(y,x)\|^2 \nonumber\\& -  \left<F(y, x), y-z^*_b(y,x)\right> + \frac{b}{2}\|y - z^*_b(y,x)\|^2 \nonumber\\
                \geq &  \frac{b-a}{2}\|y - z^*_b(y,x)\|^2.
            \end{align}
    Let us now consider  $y_s\in Y$ as the stationary point. Therefore, from Definition \ref{def:D_gap_theta} and Lemma \ref{lem:prop_paramtric_D_gap},  we have $\phi_{ab}(y_s,x) = 0.$  Now from  equation \eqref{eqn:thm1_bound} and taking into account $b>a>0$, we obtain 
    \begin{align}\label{eqn:stationary_pt_thm}
       y_s = z_b^*(y_s,x). 
    \end{align}
    This shows part (a). Now note that  the equation above is a fixed-point equation in $y$, that  is also a function of $x$. We now differentiate equation \eqref{eqn:stationary_pt_thm} and try to obtain the value for implicit gradient $\blue{\hm{\nabla}_x y}$ at the  point $y_s$. We have
    \begin{align*}
 \blue{\hm{\nabla}_x y} & =   \blue{\hm{\nabla}_x  z_b^*(y(x), x)}= \blue{\left<\hm{\nabla}_y z_b^*(y,x), \hm{\nabla}_x y\right>} +  \blue{\hm{\nabla}_x  z_b^*(y, x)}.
\end{align*}
\end{proof}

\begin{algorithm}[h]
\caption{Iterative Differentiation for Implicit Gradient}\label{alg:iterative_impl_grad}
\begin{algorithmic}[1]
\STATE{Consider $K,T\in \mathbb{N}.$  Initialize $x_0, y_0(x_0),$ stepsizes $\gamma, \beta$}
\STATE {{\bf for} $k =0, 1, 2, \dots, K$}
\STATE{\quad{\bf for} $t =0, 1, 2, \dots, T$
\begin{align}\label{eqn:alg_fixed_Pt_update}
    &\quad z^*_b(y_t; x_k) \nonumber\\& \ \quad= \underset{z\in Y}{\text{argmax}} \left\{\blue{\left<F(y_t, x_k),y_t-z\right>} - \frac{b}{2}\|y_t - z\|^2 \right\}.\nonumber\\
    &\quad y_{t+1}(x_k) := z^*_b(y_{t},x_k).
\end{align}}
\STATE\quad{\bf end for}
\STATE{Evaluate the gradient  from equation \eqref{eqn:grad_theta_1} \begin{align*}
\nabla_x f(y_{T+1}(x_k), x_k) = & \nabla_x f(y_{T+1}(x_k), x_k) \\+& \nabla_y f(y_{T+1}(x_k), x_k) \blue{\hm{\nabla}_x y_{T+1}(x_k)}.
\end{align*}}
\STATE{Update $x_{k+1} := \mathcal{P}_X\left\{ x_k - \beta\nabla_x f(y_{T+1}(x_k), x_k)\right\}.$}

\STATE{\bf end for}
\end{algorithmic}
\end{algorithm}

Next we provide a result to make sure  the gradient obtained in Theorem \ref{thm:differentiation_D-gap} exists.

\begin{lemma}
Provided Assumptions \ref{assum:problem_struct} and \ref{assum:implicit_grad_exactness} hold on the structure of problem   \eqref{prob:main_prob}, the implicit gradient provided by equation \eqref{eqn:grad_theta_1}  is  unique and exists everywhere.
\end{lemma}
\begin{proof}
   Consider the D-gap function, defined in Definition \ref{def:D_gap_theta}. The objective function in problem \eqref{prob:project_y_(c,theta)} is strongly concave. From the strong concavity and from the MFCQ condition (which is equivalent to Slater CQ for differentiable function $g$ Thm 2.3.8, 3.2.8 in \cite{Borwein00convexanalysis}, we have the  continuously differentiability of the solution map \cite{barratt2018differentiability}.   
\end{proof}

Next result discusses smoothness of the solution of the inner-level VI.

\begin{lemma}[Theorem 4.2.16 \cite{LuoPangRalph1996}]\label{lem:piecewise_smooth_S(x)}
Let the set valued map $Y(x)\equiv \{y\in \mathbb{R}^m : g_i(x,y) \leq 0 \}$ such that Assumption \ref{assum:implicit_grad_exactness} is satisfied. Now let $(x^*,y^*)$ is the solution of $\mathrm{SOL}(Y(x), F(y,x))$. Then provided necessary constraint qualifications (MFCQ, CRCQ, SCOC) are satisfied, there exists a neighborhood  $U\times V $ of $(x^*, y^*)$, such that $y: U\to V$  is and $y(x)$ is a unique solution of $\mathrm{VI}(Y(x), F(x,y))$ and the solution map $S(x)$ is smooth.  
\end{lemma} 
\begin{proof}
It can be seen that under Assumption \ref{assum:implicit_grad_exactness}, the necessary Constraint Qualifications  (MFCQ, CRCQ, and SCOC \cite{LuoPangRalph1996}, \cite{gould1971necessary}, \cite{peterson1973review})  hold. Therefore, $y(x)$ is unique and $S(x)$ is continuously differentiable. 
\end{proof}

Next result, we comments on the Lipschitz continuity of solution map $S(x)$.

\begin{lemma}\label{lem:Lipschitz_smooth_S(x)_f}
Consider problem \eqref{prob:main_prob}. The solution map of the inner-level VI, $S(x):X\to Y$   is Lipschitz continuous with respect to parameter $0<L_S<\infty$. 
\end{lemma}
\begin{proof}
From the  continuity  of the  solution map of the inner-level of problem   \eqref{prob:main_prob} (discussed in Lemma \ref{lem:piecewise_smooth_S(x)}) and from the boundedness of  set $Y$ (Assumption \ref{assum:problem_struct}(c)), there exists a scalar $({L_S}<\infty)$ such that   $\blue{\|\hm{\nabla}_x  S(x)\|}$ is bounded by ${L_S}$ for any $y\in S(x)\subseteq Y$.    
\end{proof}

\section{Error Bounds and Rate Analysis}\label{sec:converg_bound_and_rates}
In this section, we discuss the error bounds on the gradients, obtained from Algorithm \ref{alg:iterative_impl_grad} and provide the  rate results in Theorem \ref{thm:subliner_convergence}. For  further analysis we provide here another set of assumptions on the smoothness of  
$f$ in problem \eqref{prob:main_prob}. 
\begin{assumption}\label{assum:Lipschitz_func_f}
Consider problem \eqref{prob:main_prob}.  The gradient of the objective function $f(x, y)$ has the following properties:
\begin{itemize}
    \item[(a)]  We assume the Lipschitz smoothness property for  $f(\bar x, y)$ with respect to $y$, i.e. for any $\bar x \in X$, and  $y_1, y_2\in Y$, we have
    \begin{align*}
        &\|\nabla_x f(\bar x, y_1) -\nabla_x f(\bar x, y_2)\| \leq L_{f_x} \|y_1 - y_2\|\\&\quad \text{and } \
        \|\nabla_y f(\bar x, y_1) - \nabla_y f(\bar x, y_2)\| \leq L_{f_y} \|y_1 - y_2\|.
    \end{align*}
    \item[(b)] We assume the Lipschitz continuity  for  $\nabla_y f(\bar x, \bar y)$ with respect to $x$ for any $\bar y \in Y$, i.e. for any $\bar x_1, \bar x_2 \in X,$ and $y\in Y$, we have
    \begin{align*}
        \|\nabla_yf(\bar x_1, \bar y) - \nabla_yf(\bar x_2, \bar y) \| \leq  \bar L_{f_y} \|\bar x_1 - \bar x_2\|.
    \end{align*}
    \item[(c)] Function $f$ is $M$-Lipschitz. For $x_1,x_2\in X$, we have 
    \begin{align*}
        \|f(x_1) - f(x_2)\| \leq M \|x_1 - x_2\|.
    \end{align*}
\end{itemize}
\end{assumption}

In the next result, we establish  the Lipschitz smoothness constant for $f$ with respect to $x$.

\begin{lemma}\label{lem:Lipschitz_f}
Provided Assumption \ref{assum:Lipschitz_func_f} hold on the objective function of problem \eqref{prob:main_prob}. For $x_1, x_2 \in X$, we have the following
\begin{align*}
    \|\nabla_x f(x_1, y^*(x_1)) - \nabla_x f(x_2, y^*(x_2))\| \leq L_f \|x_1 - x_2\|,
\end{align*}
where $L_f\triangleq L_S (\bar{L}_{f_y}+ L_{f_\omega}).$
\end{lemma}
\begin{proof}
Consider $\|\nabla_x f(x_1, y^*(x_1)) - \nabla_x f(x_2, y^*(x_2))\|$. Using the triangle inequality, we can write this as
\begin{align*}
    &\|\nabla_x f(x_1, y^*(x_1)) - \nabla_x f(x_2, y^*(x_2))\| \\
    & \leq  \|\nabla_x f(x_1, y^*(x_1)) - \nabla_x f(x_2, y^*(x_1)) \|\\ & \quad + \|\nabla_x f(x_2, y^*(x_1)) - \nabla_x f(x_2, y^*(x_2))\|.
\end{align*}
 Next, from Assumption \ref{assum:Lipschitz_func_f}(a) and using Cauchy-Schwarz, we bound the above as
\begin{align*}
    & \|\nabla_x f(x_1, y^*(x_1)) - \nabla_x f(x_2, y^*(x_2))\| \\
     & \leq \left\|\nabla_y f(x_1, y^*(x_1)) - \nabla_y f(x_2,y^*(x_1))\right\|\left\|\blue{\hm{\nabla}_x y}\right\| \\& \quad+ L_{f_x}\| y^*(x_1) - y^*(x_2)\|.
\end{align*}
Next, from Assumption \ref{assum:Lipschitz_func_f}(b), we bound the above as
\begin{align*}
    & \|\nabla_x f(x_1, y^*(x_1)) - \nabla_x f(x_2, y^*(x_2))\| \\
    &  \leq \bar{L}_{f_y} \left\|x_1 -x_2\right\|\left\|\blue{\hm{\nabla}_x y} \right\| + L_{f_x}\| y^*(x_1) - y^*(x_2)\|.
\end{align*}
Recalling the Lipschitz continuity of solution map $S(x)$  from Lemma \ref{lem:Lipschitz_smooth_S(x)_f}, 
we have the required result.
\end{proof}

In the following result, we show that the solution obtained from fixed-point equation is a contraction mapping. 
\begin{lemma}\label{lem:contraction_map_z}
Consider the fixed-point equation obtained from Theorem \ref{thm:differentiation_D-gap}(a). Show that it is a contraction map.
\end{lemma}
\begin{proof}
Vector $z^*_b(y,x)$ is  an optimal solution, obtained by solving a skewed projection problem \eqref{prob:project_y_(c,theta)}. 
Projection operator is a nonexpansive map \cite{facchinei2003finite}. \blue{From Theorem 12.1.2 in \cite{facchinei2003finite}, as long as for the contraction coefficient ($q_x$), we have $q_x^2 \leq 2\mu$ for $\mu$-strongly monotone  $F$, we have the required result.} \end{proof}



The above result is revisited in Assumption \ref{assum:inner_level}(d) again.
Next, we will discuss main results of this work. In Theorem \ref{thm:subliner_convergence} we show that the update from Algorithm \ref{alg:iterative_impl_grad} converges to local optimum with $\mathcal{O}(1/K)$. 
Before that,  we provide next result for the convergence of the inner-level problem. We tackle this as the convergence of a sequence, generated from iteratively solving the  fixed-point equation \eqref{eqn:alg_fixed_Pt_update}.
\begin{lemma}\label{lem:fixed_pt_err_bd}
Consider problem \eqref{prob:main_prob}. We show  that for some $x\in X$,  iterative update of $y_k$, obtained from equation \eqref{eqn:alg_fixed_Pt_update} in Algorithm \ref{alg:iterative_impl_grad} converges to the limit point $y^*$ with an R-linear rate, after iteration $k$ of the inner-level loop in Algorithm \ref{alg:iterative_impl_grad}
\begin{align*}
    \|y_k - y^*\| \leq \sqrt{\frac{ \phi_{ab}(y_{0}, x)}{C_1}} \frac{1}{1-\sqrt{\frac{C_2}{C_1 + C_2}}} \left(\sqrt{\frac{C_2}{C_1 + C_2}}\right)^k,
\end{align*}
where $C_1, C_2,$ and $\delta$ are the nonnegative scalars such that for any $x\in \mathbb{R}^m$, and $y\in\mathbb{R}^n$ we have
\begin{align*}
    &\phi_{ab}(y,x) -  \phi_{ab}(z_b^*(y,x), x) \geq C_1 \|y - z^*(y,x)\|^2, \\ 
    &\min(\phi_{ab}(y,x),  \phi_{ab}(z^*_b(y,x), x)) \leq C_2 \|y - z^*_b(y,x)\|^2,\\
    &\|y - z^*_b(y,x)\|\leq \delta.
 \end{align*}
\end{lemma}
\begin{proof}
From the definitions of $C_1, C_2,$  $\delta$, we have
\begin{align*}
   & \phi_{ab}(y_k, x) - \phi_{ab}(y_{k+1}, x) \geq C_1 \|y_k - y_{k+1}\|^2,\\
    &    \phi_{ab}(y_{k+1}, x)\leq C_2 \|y_k - y_{k+1}\|^2.
\end{align*}
From the above two, we have the nonnegative sequence $\{\phi_{ab}(y_k, x) \} $  converging to zero. Therefore, we can write 
\begin{align*}
    \phi_{ab}(y_{k+1}, x) \leq \frac{C_2}{C_1 + C_2} \phi_{ab}(y_{k}, x).
\end{align*}
For sufficiently large $k$, telescoping the above equation and utilizing the bounds above, we have 
\begin{align*}
    C_1 \|y_k - y_{k+1} \|^2 \leq \phi_{ab}(y_{k}, x) \leq \left(\frac{C_2}{C_1 + C_2}\right)^k \phi_{ab}(y_{0}, x),
\end{align*}
this can be written as
\begin{align*}
    \|y_k - y_{k+m} \| \leq \sqrt{\frac{ \phi_{ab}(y_{0}, x)}{C_1}} \sum _{j=k}^{k+m-1} \left(\sqrt{\frac{C_2}{C_1 + C_2}}\right)^j.
\end{align*}
Therefore, $\{y_k\}$ is a Cauchy sequence that converges to a limit point $(y^*)$. Utilizing the continuity of function $\phi_{ab}$, we have \begin{align*}
    \|y_k - y^*\|\leq \sqrt{\frac{ \phi_{ab}(y_{0}, x)}{C_1}} \frac{1}{1-\sqrt{\frac{C_2}{C_1 + C_2}}} \left(\sqrt{\frac{C_2}{C_1 + C_2}}\right)^k.
\end{align*}\end{proof}

Next, we work on obtaining the error bound between  the implicit gradient from iterative update \eqref{eqn:alg_fixed_Pt_update} in Algorithm \ref{alg:iterative_impl_grad} and the actual implicit gradient. We start with the following assumptions on the inner-level fixed-point problem.
\begin{assumption} \label{assum:inner_level}
Consider Problem \eqref{prob:main_prob}, Assumption \ref{assum:problem_struct}, and the update obtained in equation \eqref{eqn:alg_fixed_Pt_update}. We have the following hold:
\begin{itemize}
    \item[(a)] \blue{Jacobians} $\blue{\hm{\nabla}_x z_b^*(y,x)}$ and $\blue{\hm{\nabla}_y z_b^*(y,x)}$ are Lipschitz continuous with constants $L_{x_{in}}$ and $L_{y_{in}}$, respectively.
    \item[(b)] Considering the boundedness of set $Y$ (Assumption \ref{assum:problem_struct}(c)), there exists a bound on the update from equation \eqref{eqn:alg_fixed_Pt_update}, $\|y(x)\|\leq C_{y_{in}}$.
    \item[(c)] \blue{There exists $C'_{x_{in}}>0$,}  $\underset{\|y\|\leq 2C_{y_{in}}}{\text{sup}}\left\|\blue{\hm{\nabla}_x z_b^*(y,x)}\right\|\leq C'_{x_{in}}$, where $C_{y_{in}} >0$,
    \item[(d)] Referring to Lemma \ref{lem:contraction_map_z}, we have $q_x\in (0,1)$ as the contraction coefficient for $z^*_b(\cdot, x)$ \blue{such that $q_x^2 \leq 2\mu.$}
\end{itemize}
\end{assumption}


In the next result, we will derive the error bound on  difference between the implicit gradient obtained from Algorithm \ref{alg:iterative_impl_grad} and true  gradient at solution $y^*(x)$ for the inner-level VI.

\begin{proposition}\label{proposition:implicit_grad_err_bnd}
Consider problem \eqref{prob:main_prob}. Let Assumptions  \ref{assum:problem_struct}, and \ref{assum:inner_level} hold. Then we have the  the error bound for the implicit gradient at the iterative update obtained from equation \eqref{eqn:alg_fixed_Pt_update} after iteration $T$, and  the  gradient of the inner-level fixed-point of the VI in problem \eqref{prob:main_prob} as follows
\begin{align*}
    \left\|\blue{\hm{\nabla}_x y_T} -\blue{\hm{\nabla}_x y^*}\right\| \leq &\left(L_{x_{in}} + \frac{L_{y_{in}}C'_{x_{in}}}{1-q_x}\right)C_{y_{in}}q_x^{T-1}T\\ &+ \frac{C'_{x_{in}}}{1-q_x}q_x^{T}.
\end{align*}
\end{proposition}
\begin{proof}
Consider equation \eqref{eqn:alg_fixed_Pt_update}. Differentiating $y_{T-1} = z^*_b(y_T, x)$ we have the following at $y_T$ and $y^*$.
\begin{align*}
    \blue{\hm{\nabla}_x y_T}  &= \left<\blue{\hm{\nabla}_y z_b^*(y_{T-1}, x)}, \blue{\hm{\nabla}_x y_{T-1}}\right> +  \blue{\hm{\nabla}_x z_b^*(y_{T-1}, x)}\\
    \blue{\hm{\nabla}_x y^*}  &= \left<\blue{\hm{\nabla}_y z_b^*(y^*, x)}, \blue{\hm{\nabla}_x y^*}\right> +  \blue{\hm{\nabla}_x z_b^*(y^*, x)}.
\end{align*}
Substituting the above in $  \left\|\blue{\hm{\nabla}_x  y_T} - \blue{\hm{\nabla}_x y^*}\right\|$, we have
\begin{align*}
   &\hspace{-0.5cm} \left\|\blue{\hm{\nabla}_x  y_T} - \blue{\hm{\nabla}_x y^*}\right\|
    \\ \leq &\left\| \blue{\hm{\nabla}_y z_b^*(y_{T-1}, x)}  +  \blue{\hm{\nabla}_x  z_b^*(y_{T-1}, x) }\right\| \left\|\blue{\hm{\nabla}_x  y^*}\right\|\\& +\left\| \blue{\hm{\nabla}_y z_b^*(y_{T-1}, x)}\right\|\left\|\blue{\hm{\nabla}_x  y_{T-1}}-\blue{\hm{\nabla}_x  y^*}\right\|\\
    &+\left\|\blue{\hm{\nabla}_x z_b^*(y_{T-1}, x)} -\blue{\hm{\nabla}_x  z_b^*(y^*, x)}\right\|.
    \end{align*}
Next, from Assumption \ref{assum:problem_struct} and \ref{assum:inner_level}, we bound the above as
\begin{align*}
    \left\|\blue{\hm{\nabla}_x  y_T} - \blue{\hm{\nabla}_x  y^*}\right\| \leq& \left(L_{x_{in}} + \frac{L_{y_{in}}C'_{x_{in}}}{1-q_x}\right)\|y_{T-1} - y^*\|\\& + q_x \left\|\blue{\hm{\nabla}_x  y_{T-1} }-\blue{\hm{\nabla}_x  y_T}\right\|.
   \end{align*}
   Next, utilizing a result on the recursive error bound from  Lemma 1, Section 2.2 in \cite{Polyak1987}  and using the following bounds, we establish the required result on error bound. 
   \begin{align*}
       &\|y^* - y_0\| = \|y^*\| \leq C_{y_{in}}, \\ 
       &\left\|\blue{\hm{\nabla}_x  y^*} -\blue{\hm{\nabla}_x  y_0}\right\|  \leq \left\| \blue{\hm{\nabla}_x y^*}\right\| \leq \frac{C'_{x_{in}}}{1-q_x}.
       \end{align*}
 \end{proof}
Next, we will discuss one of the main results of this work. We show that the update from Algorithm \ref{alg:iterative_impl_grad} converges to local optimum with $\mathcal{O}(1/K)$. 

\begin{theorem}\label{thm:subliner_convergence}
Consider problem \eqref{prob:main_prob}. Let Assumption \ref{assum:problem_struct},  \ref{assum:Lipschitz_func_f}, and \ref{assum:inner_level} hold.  Consider the update  from step 6 of Algorithm \ref{alg:iterative_impl_grad}. We show that sequence $\{x_k\}$ converges to local optimal solution with a rate $\mathcal{O}(1/K)$ for $K$ iterations 
\begin{align*}
    & \hspace{-0.25cm}\underset{k\in\{0,\dots, K\}}{\text{min}}\| \nabla_x f(y^*(x_k), x_k) \|^2 \\ 
    &\leq  \frac{f(y^*(x_0), x_0) -f (y^*(x_{K+1}), x_{K+1})}{\beta\left( \frac{1}{2} - \beta L \right) K}\\ & + L_f(1+L_S)\sqrt{\frac{ \phi_{ab}(y_{0})}{C_1}} \frac{\left( \frac{\beta}{2} + \beta^2L_f \right)}{1-\sqrt{\frac{C_2}{C_1 + C_2}}} \left(\sqrt{\frac{C_2}{C_1 + C_2}}\right)^{T+1} \\
     &+ M\left(\frac{\beta}{2} + \beta^2L_f\right) \left(\left(L_{x_{in}} + \frac{L_{y_{in}}C'_{x_{in}}}{1-q_x}\right)C_{y_{in}}q_x^{T}(T+1) \right.\\&\hspace{6cm}\left.+ \frac{C'_{x_{in}}}{1-q_x}q_x^{T+1}\right).
\end{align*}
\end{theorem}
\begin{proof}
Consider problem \eqref{prob:main_prob}. The total gradient of the objective function is
\begin{align*}
    \nabla_x f(y_{T+1}(x_k), x_k) =& \nabla_x f(y_{T+1}(x_k), x_k)\\ + &\left<\nabla_y f(y_{T+1}(x_k), x_k), \blue{\hm{\nabla}_x y_{T+1}(x_k)}\right>\\
    \nabla_x f(y^*(x_k), x_k) =& \nabla_x f(y^*(x_k), x_k)\\& + \left<\nabla_y f(y^*(x_k), x_k),\blue{\hm{\nabla}_x y^*(x_k)}\right>.
\end{align*}

 Using the Lipschitz smoothness of $f$, we have
\begin{align*}
    \|\nabla_x&f(y_{T+1}(x_k), x_k)  -\nabla_x f(y^*(x_k), x_k)\|  \\
   \leq &L_f\|y_{T+1}(x_k) - y^*(x_k)\| \\&+ M \left\|\blue{\hm{\nabla}_x  y_{T+1}(x_k)} - \blue{\hm{\nabla}_x  y^*(x_k)}\right\|\\
    &+L_f\underbrace{\left\|  \blue{\hm{\nabla}_x  y^*(x_k)}\right\|}_{\text{term 1}}\|y_{T+1}(x_k) - y^*(x_k)\| 
\end{align*}
From the boundedness of set $Y$ and a continuity of  $ y^*(x_k)$ over set $Y$, we have bound on term 1 as $L_S$. Above equation can be written as
\begin{align*}
    \|\nabla_x&f(y_{T+1}(x_k), x_k)  -\nabla_x f(y^*(x_k), x_k)\|  \\
   \leq &L_f(1+L_S)\underbrace{\|y_{T+1}(x_k) - y^*(x_k)\|}_{\text{term 2}} \\&\qquad\qquad+ M \underbrace{\left\|\blue{\hm{\nabla}_x  y_{T+1}(x_k)} -  \blue{\hm{\nabla}_x y^*(x_k)}\right\|}_{\text{term 3}}.
\end{align*}
Next, we bounds terms 2 and 3 in the above from the results in Lemma \ref{lem:fixed_pt_err_bd} and Proposition \ref{proposition:implicit_grad_err_bnd}, we have
\begin{align*}
    &\|\nabla_xf(y_{T+1}(x_k), x_k)  -\nabla_x f(y^*(x_k), x_k)\|  \\
   \leq &L_f(1+L_S)\sqrt{\frac{ \phi_{ab}(y_{0}, x_k)}{C_1}} \frac{1}{1-\sqrt{\frac{C_2}{C_1 + C_2}}} \left(\sqrt{\frac{C_2}{C_1 + C_2}}\right)^{T+1} \\&\qquad\qquad+ M \left(\left(L_{x_{in}} + \frac{L_{y_{in}}C'_{x_{in}}}{1-q_x}\right)C_{y_{in}}q_x^{T}(T+1) \right.\\&\qquad\qquad\qquad\qquad\left.+ \frac{C'_{x_{in}}}{1-q_x}q_x^{T+1}\right).
\end{align*}
Next, taking into account the Lipschitz smoothness of the objective function (Lemma \ref{lem:Lipschitz_f}) for problem \eqref{prob:main_prob}, we have the following for any two $x_{k}, x_{k+1} \in X$
\begin{align*}
     f& (y^*(x_{k+1}), x_{k+1})\leq f(y^*(x_k), x_k)\\ &+\left< \nabla_x f( y^*(x_k), x_k),  x_{k+1}  -  x_{k} \right> + \frac{L_f}{2}\|x_{k+1}  -  x_{k}\|^2.
\end{align*}
Substituting the update rule from Algorithm \ref{alg:iterative_impl_grad}, utilizing the nonexpansiveness of the projection mapping, Cauchy-Schwarz inequality,  and adding subtracting $\beta\|\nabla_x f(y^*(x_k), x_k)\|$, we obtain
\begin{align*}
     f& (y^*(x_{k+1}), x_{k+1})\leq f(y^*(x_k), x_k)\\ & - \left( \frac{\beta}{2} - \beta^2L_f \right)\| \nabla_x f(y^*(x_k), x_k)  \|^2 \\ &+   \left( \frac{\beta}{2} + \beta^2L_f \right)\underbrace{\| \nabla_x f(y^*(x_k), x_k) - \nabla_x f(y_{T+1}(x_k), x_k)  \|^2}_{\text{term 4}}.
\end{align*}
Substituting the bound for term 4, we have
\begin{align*}
    & f (y^*(x_{k+1}), x_{k+1})\leq f(y^*(x_k), x_k)\\ & - \left( \frac{\beta}{2} - \beta^2L_f \right)\| \nabla_x f(y^*(x_k), x_k)  \|^2 \\ 
     & + L_f(1+L_S)\sqrt{\frac{ \phi_{ab}(y_{0}, x_k)}{C_1}} \frac{\left( \frac{\beta}{2} + \beta^2L_f \right)}{1-\sqrt{\frac{C_2}{C_1 + C_2}}} \left(\sqrt{\frac{C_2}{C_1 + C_2}}\right)^{T+1} \\
     &+ M\left(\frac{\beta}{2} + \beta^2L_f\right) \left(\left(L_{x_{in}} + \frac{L_{y_{in}}C'_{x_{in}}}{1-q_x}\right)C_{y_{in}}q_x^{T}(T+1) \right.\\&\qquad\qquad\qquad\qquad\left.+ \frac{C'_{x_{in}}}{1-q_x}q_x^{T+1}\right).
\end{align*}
Taking summation on both sides over $k$ from 0 to $K$, we have
\begin{align*}
    & \hspace{-0.25cm}\underset{k\in\{0,\dots, K\}}{\text{min}}\| \nabla_x f(y^*(x_k), x_k) \|^2 \\ 
    &\leq  \frac{f(y^*(x_0), x_0) -f (y^*(x_{K+1}), x_{K+1})}{\beta\left( \frac{1}{2} - \beta L \right) K}\\ & + L_f(1+L_S)\sqrt{\frac{ \phi_{ab}(y_{0}, x_k)}{C_1}} \frac{\left( \frac{\beta}{2} + \beta^2L_f \right)}{1-\sqrt{\frac{C_2}{C_1 + C_2}}} \left(\sqrt{\frac{C_2}{C_1 + C_2}}\right)^{T+1} \\
     &+ M\left(\frac{\beta}{2} + \beta^2L_f\right) \left(\left(L_{x_{in}} + \frac{L_{y_{in}}C'_{x_{in}}}{1-q_x}\right)C_{y_{in}}q_x^{T}(T+1) \right.\\&\qquad\qquad\qquad\qquad\left.+ \frac{C'_{x_{in}}}{1-q_x}q_x^{T+1}\right).
\end{align*}
Note that the last two terms in the above go to zero with increasing number of inner iteration $T$. We hereby focus on establishing the nonasymptotic  convergence rate of the outer-level update $\{x_k\}$ from Algorithm \ref{alg:iterative_impl_grad}. Therefore, assuming the inner-level converges R-linearly, we bound the last two terms with $\epsilon$ and we secure the rate  of $\mathcal{O}\left(\frac{1}{K}\right)$.
\end{proof}

\section{Conclusion}
In this work, we study a class of optimization problems with variational inequality (VI) constraints, motivated by a wide range of applications in machine learning and large-scale structured systems. We propose an implicit gradient scheme based on the iterative differentiation (ITD) strategy, offering an efficient approach for computing the implicit gradient required to evaluate the outer-level objective function. The proposed scheme is particularly suitable for large-scale structures, where direct differentiation or matrix inversion becomes computationally expensive. We further establish error bounds with respect to the true gradients. Moreover, under a nonconvex objective and a strongly monotone mapping, we analyze the non-asymptotic convergence properties of the proposed method and derive corresponding rate results.

\bibliography{sample}

\end{document}